\documentclass[a4paper,12pt]{article}

\usepackage{amsmath, amsfonts, amsthm,mathrsfs, amssymb, xfrac, color} 
\usepackage[margin=1in]{geometry}

 \newtheorem{thm}{Theorem}[section]
 
 \newtheorem{lem}[thm]{Lemma}
 \newtheorem{prop}[thm]{Proposition}
 \theoremstyle{definition}
 
 \theoremstyle{remark}
 \newtheorem{rem}[thm]{Remark}
 
 \numberwithin{equation}{section}

\newcommand{\nm}[1]{\|#1\|}								
\newcommand{\ip}[1]{\langle #1 \rangle}						
\newcommand{\inftysm}[1]{\sum_{#1}^{\infty}}					
\newcommand{\sm}[2]{\sum_{#1}^{#2}}						
\newcommand{\eqns}[1]{\begin{align*} #1 \end{align*}}		
\newcommand{\eqnslab}[1]{\begin{align} #1 \end{align}}		
\newcommand{\nn}{\nonumber}
\newcommand{\bracks}[1]{\left( #1 \right)}						
\newcommand{\curlybracks}[1]{\left\{ #1 \right\}}					
\newcommand{\modu}[1]{\left| #1 \right|}						

\newcommand{\Df}{\mathscr{D}_f}                                               
\newcommand{\Da}{\mathscr{D}_\alpha}

\newcommand{\fT}{\mathscr{T}}                     

\newcommand{\Mf}{\mathscr{M}_f}                                    

\newcommand{\Mtwo}{\mathcal{M}^2}
\newcommand{\Mctwo}{\mathcal{M}_c^2}
\newcommand{\Mcp}{\mathcal{M}_c^p}
\newcommand{\Mp}{\mathcal{M}^p}

\newcommand{\Mc}{\mathcal{M}_c}
\newcommand{\elp}{\ell^p}
\newcommand{\elq}{\ell^q}
\newcommand{\eltwo}{\ell^2}

\newcommand{\el}{\ell}

\newcommand{\N}{\mathbb{N}}
\newcommand{\C}{\mathbb{C}}
\newcommand{\R}{\mathbb{R}}

\newcommand{\Q}{\mathbb{Q}}
\newcommand{\Prime}{\mathbb{P}}
\newcommand{\Qplus}{\mathbb{Q^+}}

\newcommand{\conv}[2]{\bracks{#1 \ast #2}}

\begin{document}

\title{Bounded multiplicative Toeplitz operators \\ on sequence spaces}

\author{Nicola Thorn}
\date{}
\newcommand{\addressprint}{ Nicola Thorn \\ Department of Mathematics \\ University of Reading\\ Whiteknights \\ PO Box 22\\ Reading RG66AX\\
UK \\ email: n.j.b.thorn@pgr.reading.ac.uk}



\maketitle

\begin{abstract}
In this paper, we study the linear mapping which sends the sequence  $x=(x_n)_{n \in \N}$ to $y=(y_n)_{n \in \N}$ where $y_n = \inftysm{k=1}f(n/k)x_k$ for $f: \Qplus \to \C$. This operator is the multiplicative analogue of the classical Toeplitz operator, and as such we denote the mapping by $\Mf$. We show that for $1 \leq p \leq q \leq \infty$, if $f \in \el^r(\Qplus)$, then $\Mf:\elp \to \elq$ is bounded where $\frac{1}{r} = 1 - \frac{1}{p} + \frac{1}{q} $. Moreover, for the cases when $p=1$ with any $q$, $p=q$, and $q=\infty$ with any $p$, we find that the operator norm is given by $\nm{\Mf}_{p,q} = \nm{f}_{r,\Qplus}$ when $f \geq 0$. Finding a necessary condition and the operator norm for the remaining cases highlights an interesting connection between  the operator norm of $\Mf$ and elements in $\elp$ that have a multiplicative structure, when considering $f:\N \to \C$. We also provide an argument suggesting that $f \in \el^r$ may not be a necessary condition for boundedness when $1<p<q<\infty$.\noindent
\medskip

\textbf{Keywords:} bounded multiplicative Toeplitz operators, multiplicative sequences, sequence spaces
\noindent

\textbf{MSC (2010):} Primary 47B37; Secondary 47B35, 11N99
\end{abstract}

\section{Introduction}
  \label{S:1}

  In this paper, we study the multiplicative Toeplitz operator, denoted by $\Mf$, which sends a sequence $(x_n)_{n \in \N}$ to $(y_n)_{n \in \N}$ where
  \eqnslab{
  y_n = \inftysm{k =1} f\bracks{\frac{n}{k}}x_k,
\label{eqn:mapping}
  }
  and $f$ is a function defined from the positive rationals, $\Qplus$, to $\C$. We can think of $\Mf$ as being given by the infinite matrix $A_f$ whose entries are $a_{i,j} = 
  f(i/j)$ for $i,j \in \N$:

  \[
  A_f=
    \begin{pmatrix}
      f(1) & f(1/2) & f(1/3) & f(1/4) & \cdots \\
      f(2) & f(1) & f(2/3) & f(1/2) &  \cdots\\
      f(3) & f(3/2) & f(1) & f(3/4) & \cdots \\
      f(4) & f(2) & f(4/3) & f(1) & \cdots \\
      f(5) & f(5/2) & f(5/3) & f(5/4) & \cdots \\
      \vdots & \vdots & \vdots & \vdots & \ddots
    \end{pmatrix}
  \]
Characterised by matrices with constants on skewed diagonals, these mappings are the ``multiplicative" analogue of the vastly studied classical Toeplitz operators on sequence spaces. The topic of multiplicative analogues of Toep\-litz operators,  discussed in  \cite{Hilberdink2009}, \cite{Hilberdink2015} and \cite{codeca2008} for example, has grown in recent years, with the study of other multiplicative constructions; for example, \cite{Brevig2016} and  \cite{Perfekt2016} investigate the multiplicative Hankel matrix, otherwise known as Helson matrices.

Toeplitz operators, $\fT_\phi$, are most often studied via the function $\phi$, which is referred to as the symbol. In a similar manner, we shall be considering $\Mf$ in terms of the function $f$ and asking for which $f$ do certain properties hold\footnote{The symbol of $\Mf$ would be given by $F(t) = \sum_{q \in \Q} f(q)q^{it}$ where $t \in \R$}. By taking $f$ supported only on $\N$, we have $y_n = \sum_{d \mid n} f\bracks{\frac{n}{d}}x_d = \conv{f}{x}(n)$ where $\ast$ is Dirichlet convolution \cite{Apostol1976}. In this case, $A_f$ becomes a lower triangular matrix given by
  \eqns{
  A_f=
    \begin{pmatrix}
      f(1) & 0 & 0 & 0 & \cdots \\
      f(2) & f(1) & 0 & 0 &  \cdots\\
      f(3) & 0 & f(1) & 0 & \cdots \\
      f(4) & f(2) & 0 & f(1) & \cdots \\
      f(5) & 0 &0 & 0 & \cdots \\
      \vdots & \vdots & \vdots & \vdots & \ddots
    \end{pmatrix}
    }
We shall denote the mapping induced by this matrix by $\Df$.

Interesting connections to analytic number theory and many open questions have fuelled recent research. For example, in \cite{Hilberdink2009} the author illustrates a connection between these operators and the Riemann zeta function. Namely, by choosing $f$ to be supported on $\N$ where $f(n) = \frac{1}{n^\alpha}$ (denoted by $\Da$),  we have that $\Da : \eltwo \to \eltwo$ is bounded $\iff\alpha > 1 $, in which case $\nm{\Da}_{2,2 }= \zeta(\alpha)$. Thus when $\alpha \leq 1$, then $\Da$ is unbounded. By restricting the range of the mapping when $\alpha \in (\frac{1}{2},1]$ and considering
\eqns{
Y_\alpha(N) = \sup_{\nm{x}_2 = 1} \bracks{\sm{n=1}{N} \modu{y_n}^2}^{\frac{1}{2}},
}
it can be shown that $Y_\alpha (N)$ is a lower bound for the maximal order of the Riemann zeta function. Specifically, for $\alpha \in (\frac{1}{2},1)$
\eqns{
Z_{\alpha}(T) = \max_{t \in [0,T]} \modu{\zeta(\alpha + it)} \geq Y_\alpha \bracks{T^{2/3(\alpha - 1/2) - \epsilon}},
}
for sufficiently large $T$. Moreover, an estimate for $Y_\alpha(N)$ leads to
\eqns{
\log Z_\alpha(T) \gg \frac{(\log T)^{1-\alpha}}{\log \log T},
} a known estimate for the maximal order of $\zeta$. There have since been some improvements upon this estimate, and new estimates for the case when $\alpha = \frac{1}{2}$ have been found, which interestingly utilise a similar method \cite{Aistleitner2016}, \cite{Seip2017}. For other literature on the connections to the Riemann zeta function see also \cite{Hilberdink2013}, \cite{Hilberdink2015}.

The authors of \cite{codeca2008} also highlight an application of analytic number theory to these operators, by using the properties of smooth numbers to ascertain $\nm{\Df x}_{p,p} = \nm{f}_1$ when $f$ is expressible in terms of completely multiplicative and non-negative functions (see the preliminaries for definitions).

One can also consider the matrix properties of these mappings. For example, \cite{Hilberdink2006} considers the determinants of multiplicative Toeplitz matrices. By taking an $N \times N$ truncation, denoted by $A_f(N)$, the author is able to show that if $f$ is multiplicative, then the determinant of $A_f(N)$ can be given as a product over the primes up to $N$, of determinants of Toeplitz matrices.

In Section 2, we generalise results on the boundedness of $\Df$ contained in \cite{Hilberdink2009} and \cite{codeca2008}, giving a partial criterion for $\Mf$ to be bounded as a mapping from $\elp \to \elq$. In an attempt to find a full criterion, we present a relationship between the sets of multiplicative sequences and the operator norm $\nm{\Df}_{p,q}$ in Section 3.  By considering $\Df$ acting upon these subsets, we are able to give a further boundedness result which, due to this connection, indicates that the extension of the partial criterion may not hold. As such, we speculate whether the result can be generalised to $\Mf$ acting on $\elp$ spaces, which is then followed by a discussion on the existence of a possible counterexample to this generalisation. We end the paper with a summary of the open problems that arise within this paper, and also some unanswered questions which are concerned with other operator properties of multiplicative Toeplitz operators such as the spectral points of $\Mf$.

\subsection*{Preliminaries and notation}

\medskip
\noindent
\textbf{Sequences and arithmetic functions.} We use the terms ``sequences" (real or complex valued) and ``functions" interchangeably, as we can write any arithmetical function $f(n)$ as a sequence indexed by the natural numbers $f = (f_n)_{n \in \N}$.

\smallskip
  \noindent
  \textbf{Multiplicative functions.}  First, we say that $f$ (not identically zero) is multiplicative if  $f(nm) = f(n)f(m)$ for every $n,m \in \N$ such that $(n,m)=1$. Secondly, we say $f$ is completely multiplicative if this holds for all $n,m \in \N$. Finally, if $g(n) = cf(n)$ where $f$ is multiplicative, we call $g$ constant multiplicative.

\smallskip
   \noindent
   \textbf{Euler products.}
   If $f$ is multiplicative such that $\sum_{n \in \N} \modu{f(n)} < \infty$, then
   \eqns{
     \inftysm{n=1} f(n) = \prod_{t \in \Prime} \inftysm{k=1} f(t^k),
   }
   where $\Prime$ is the set of prime numbers. If $f$ is completely multiplicative, we can write \eqns{\inftysm{n=1}f(n) = \prod_{t \in \Prime} \frac{1}{1-f(t)}.}

\smallskip
  \noindent
  \textbf{GDC and LCM.} We use $(n,m)$ and $[n,m]$ to denote the greatest common divisor and least common multiple of $n$ and $m$ in $\N$, respectively.
  We let $d(n)$ stand for the number of divisors of $n$, including $1$ and $n$ itself.

\smallskip
\noindent
\textbf{O-notation.} We say that $f$ is of the order of $g$ and write $f = O(g)$ if, for some constant, $\modu{f(n)} \leq C \modu{g(n)}$ as $n \to \infty$. We also write $f \ll g$ to mean $f = O(g).$

\smallskip
\noindent
\textbf{Sequence spaces.}  For $p \in [1,\infty]$, let $\elp$ denote the usual  space of sequences $x =(x_n)_{n \in \N}$ for which the norm  $\nm{x}_p := \bracks{\inftysm{n=1} \modu{x_n}^p}^{1/p}$ converges or $\nm{x}_{\infty} = \sup_{n \in \N} \modu{x_n}$ exists (if $p \in [1,\infty) \text{ or } p = \infty$ respectively). We define $\elp(\Qplus)$ to be the space of sequences $x=(x_s)_{s \in \Qplus}$ for which $\nm{x}_{p,\Qplus} = (\sum_{s \in \Qplus} \modu{x_s}^p)^{1/p}$ converges or $\nm{x}_{\infty, \Qplus} = \sup_{s \in \Qplus} \modu{x_s}$ exists. For the case when $p=2$, we also have that $\ip{x,y} = \sum_{n \in \N} x_n \overline{y_n}$.

\smallskip
\noindent
\textbf{Operator norm.}  Given a bounded linear operator $L$, we use the usual notation $\nm{L}_{p,q}$ to denote the operator norm of $L : \elp \to \elq$ which is given by $\nm{L}_{p,q} = \sup_{\nm{x}_p = 1} \nm{L x}_q$.

   \section{Partial criterion for boundedness}
  \label{S:2}
The following results extend theorems contained in \cite{Hilberdink2009} and \cite{codeca2008}.
     \begin{thm} \label{thm:boundedonlp} For $1 \leq p \leq q \leq \infty$, define $r \in [1,\infty]$ by \[\frac{1}{r} = 1 - \frac{1}{p} + \frac{1}{q}\] where $\frac{1}{\infty} = 0$. If $f \in \el^r(\Qplus)$ then $\Mf: \elp \to \elq$ is bounded. More precisely, we have
                   \eqns{
                  \nm{\Mf x}_q \leq \nm{x}_p \nm{f}_{r,\Qplus}.
                  }
                  \end{thm}

  Theorem \ref{thm:boundedonlp} gives a partial criterion for boundedness between $\elp$ and $\elq$; partial in the sense that $f \in \el^r(\Qplus)$ is a sufficient condition. It is natural to ask whether this is also a necessary condition, i.e., does $\Mf:\elp \to \elq$ bounded imply that $f \in \el^r$? Moreover, can we find the operator norm, $\nm{\Mf}_{p,q}$? For $f$ positive,  both of these questions can be answered by Theorem \ref{thm:operatornormedgecases} for the cases where $p=q$, $p=1$ with any $q$, and $q=\infty$ with any $p$.  We refer to these as the ``edge" cases.

                  \begin{thm}
  \label{thm:operatornormedgecases}
   Let us define $r$ as in Theorem \ref{thm:boundedonlp}. For $p=q$, $p=1$ {\rm (}any $q${\rm )}, $q=\infty$ {\rm (}any $p${\rm )} with $f \in \el^r(\Qplus)$ positive, we have
              \eqns{
  						\nm{\Mf}_{p,q} = \nm{f}_{r,\Qplus} .
  					}
                    \end{thm}

  \begin{proof}[Proof of Theorem \ref{thm:boundedonlp}]
  Let $y_n$ be given by (\ref{eqn:mapping}). The proof proceeds by considering separate cases.

\medskip  				
\noindent	
$\bullet$  $1 \leq p \leq q < \infty$

\medskip
\noindent		
                  By H{\"o}lder's inequality,
  				\eqns{
  							\modu{y_n} &\leq \inftysm{k=1}\modu{f\bracks{\frac{n}{k}} x_k} = \inftysm{k=1} \modu{f\bracks{\frac{n}{k}}}^{r\bracks{1-\frac{1}{p}}} \modu{f\bracks{\frac{n}{k}}}^{\frac{r}{q}} \modu{x_k}^{\frac{p}{q}} \modu{x_k}^{1-\frac{p}{q}} \\
                              &\leq  \bracks{\inftysm{k=1} \modu{f\bracks{\frac{n}{k}}}^r}^{\bracks{1-\frac{1}{p}}} \bracks{\inftysm{k=1} \modu{x_k}^p}^{\frac{1}{p}-\frac{1}{q}} \bracks{\inftysm{k=1} \modu{f\bracks{\frac{n}{k}}}^{r} \modu{x_k}^{p}}^\frac{1}{q} \\
                              &\leq \nm{f}_{r,\Qplus}^{r\bracks{1-\frac{1}{p}}} \nm{x}_p^{1-\frac{p}{q}} \bracks{\inftysm{k=1} \modu{f\bracks{\frac{n}{k}}}^{r} \modu{x_k}^{p}}^\frac{1}{q}.
                                  }
Hence,
                              \eqns{
                              \inftysm{n=1} \modu{y_n}^q \leq \nm{f}_{r,\Qplus}^{rq\bracks{1-\frac{1}{p}}} \nm{x}_p^{q-p} \inftysm{n=1}\inftysm{k=1} \modu{f\bracks{\frac{n}{k}}}^{r} \modu{x_k}^{p}.
                              }
  Considering only the summation on the RHS above,
                              \eqns{
                              \inftysm{n=1}\inftysm{k=1} \modu{f\bracks{\frac{n}{k}}}^{r} \modu{x_k}^{p} \leq  \sum_{s \in \Qplus}  \modu{f(s)}^r \inftysm{k=1}  \modu{x_k}^p   = \nm{f}_{r,\Qplus}^r \nm{x}_p^p .        }
  Therefore,
  \eqns{
  \nm{\Mf x}_q^q = \inftysm{n=1} \modu{y_n}^q \leq \nm{f}_{r,\Qplus}^{qr\bracks{1-\frac{1}{p}} + r} \nm{x}_p^{q-p+p} = \nm{f}_{r,\Qplus}^{q} \nm{x}_p^{q}.
  }
  				$\bullet$ $p=1$ and $q = \infty$ (so $r = \infty $)

\medskip
\noindent	
                  By the triangle inequality,
                  \eqns{
  							\modu{y_n} \leq \inftysm{k=1}\modu{f\bracks{\frac{n}{d}} x_k} \leq \nm{f}_{\infty,\Qplus} \inftysm{k =1} \modu{x_k} \leq \nm{f}_{\infty,\Qplus} \nm{x}_1 .
  							     }	
                  Hence,
                  $
                   \nm{\Mf x}_\infty \leq \nm{f}_{\infty,\Qplus} \nm{x}_1.
                  $

                  \medskip
                  \noindent
                 $\bullet$ $q = \infty$ with $1 < p < \infty$ (so $r=\frac{p}{p-1}$)

\medskip
\noindent
                By H{\"o}lder's inequality, we have
  					\eqns{
  									\modu{y_n} \leq \inftysm{k=1}\modu{f\bracks{\frac{n}{k}} x_k} \leq \bracks{\inftysm{k=1} \modu{f\bracks{\frac{n}{k}}}^r}^{\frac{1}{r}} \bracks{\inftysm{k=1}\modu{x_k}^p}^{\frac{1}{p}} \leq \nm{f}_{r,\Qplus} \nm{x}_p. 								
  					}
                      Thus,
                      $
                      \nm{\Mf x}_\infty \leq \nm{f}_{r,\Qplus} \nm{x}_p.$

  \medskip
\noindent
                      $\bullet$ $p =q =\infty$ (so $r=1$)

\medskip
\noindent	
We now have
                                          $
                      \modu{y_n} \leq \nm{x}_\infty \inftysm{k=1}\modu{f\bracks{\frac{n}{k}}}\leq \nm{x}_\infty \nm{f}_{1,\Qplus},
                      $
                        which gives the desired inequality
                      $\nm{\Mf x}_\infty \leq \nm{x}_\infty \nm{f}_{1,\Qplus}.$
                      \end{proof}

\begin{proof}[Proof of Theorem \ref{thm:operatornormedgecases}]
We consider each edge case separately.

\medskip
\noindent
1. We first embark on the case when $p=1$ with any $q$.

\medskip
\noindent
$\bullet$ Let $q \in [1,\infty)$, so that $r=q$.

\medskip
\noindent
Fix $c \in \N$ and let $x_n = 1$ if $n=c$ and $0$ otherwise. Then $\nm{x}_1 = 1$ and so,
\eqns{
\modu{y_n}^q = \modu{\inftysm{k=1} f\bracks{\frac{n}{k}} x_k }^q = \modu{f\bracks{\frac{n}{c}}}^q.
}
Therefore,
\eqnslab{
\nm{\Mf x}_q^q &= \inftysm{n=1} \modu{y_n}^q = \inftysm{n=1} \modu{f\bracks{\frac{n}{c}}}^q = \sum_{d \mid c} \inftysm{\substack{n =1 \\ (n,c) = d}}\modu{f\bracks{\frac{n}{c}}}^q \nn \\
&= \sum_{d \mid c} \inftysm{\substack{m=1 \\ (m,\frac{c}{d})=1}} \modu{f\bracks{\frac{md}{c}}}^q \text{ by writing }n = md \nn \\
&= \sum_{d \mid c} \inftysm{\substack{m=1 \\ (m,d)=1}} \modu{f\bracks{\frac{m}{d}}}^q \text{ by writing }\frac{c}{d} \mapsto d.  \label{eqn:sumoverdenomc}
}
Note that we can write
\eqnslab{
\label{eqn:sumoverrationals}
\nm{f}_{q,\Qplus}^q = \sum_{s \in \Qplus}\modu{f\bracks{s}}^q =\inftysm{v=1} \inftysm{\substack{u =1 \\ (u,v) = 1}}\modu{f\bracks{\frac{u}{v}}}^q.
}
By computing the difference between (\ref{eqn:sumoverrationals}) and (\ref{eqn:sumoverdenomc}), we shall show that $\nm{\Mf x}_q$ can be made arbitrarily close to  $\nm{f}_{q,\Qplus}$. We have
 \eqns{
\inftysm{v=1} \inftysm{\substack{u =1 \\ (u,v) = 1}}\modu{f\bracks{\frac{u}{v}}}^q - \sum_{d \mid c}  \inftysm{\substack{n =1 \\ (n,c) = d}}\modu{f\bracks{\frac{n}{c}}}^q  = \sum_{\substack{u,v \in \N \\ (u,v) = 1 \\ v \nmid c}}\modu{f\bracks{\frac{u}{v}}}^q.
}
Now, choose $c = (2 \cdot 3 \cdot 5 \cdots T )^k$ where $k \in \N$ and $T$ is prime. Then if $v\nmid c \implies v > T$ for $k$ large enough. Therefore, for every $\epsilon > 0$, we can choose $T$ such that
\eqns{
\nm{f}_{q,\Qplus}^q - \nm{\Mf x}_q^q = \sum_{\substack{u,v \in \N \\ (u,v) = 1 \\ v  \nmid c }} \modu{f\bracks{\frac{u}{v}}}^q <  \epsilon.
}
Hence, $\nm{\Mf}_{1,q} = \nm{f}_{q,\Qplus}$ as required.

\medskip
\noindent
$\bullet $ Let $q=\infty$, so $r = q = \infty$.

\medskip
\noindent
 Fix $c \in \N$. Like before, choose $x_n = 1$ if $n=c$ and $0$ otherwise. Again $\nm{x}_1 = 1$. Now,
\eqns{
\nm{\Mf x}_\infty = \sup_{n \in \N} \modu{y_n} = \sup_{n \in \N} \modu{f\bracks{\frac{n}{c}}}.
}
Note here that there exist $u,v \in \N$ with $(u,v) = 1$ such that  $\nm{f}_{\infty,\Qplus} - \epsilon <  \modu{f\bracks{\frac{u}{v}}} $. Choose $n = u$ and $c = v$. Then
\eqns{
\nm{f}_{\infty,\Qplus}  - \nm{\Mf x}_\infty < \epsilon.
}


\medskip
\noindent
2. Now consider the edge case where $p=q$.

\medskip
\noindent
$\bullet$  Let $1 < p =q < \infty$ so $r=1$.

\medskip
\noindent
 Fix $c \in \N$. Choose $x_n = \frac{1}{d(c)^\frac{1}{q}}$ if $n \mid c$ and $0$ otherwise. Hence, $\nm{x}_q^q = \frac{1}{d(c)}\sum_{d \mid c} 1 = 1$. By H{\"o}lder's inequality,
\eqns{
\inftysm{n=1} x_n^{q-1} y_n \leq \bracks{\inftysm{n=1} \modu{x_n}^q}^{1-\frac{1}{q}} \bracks{\inftysm{n-1}y_n^q}^\frac{1}{q}  = \bracks{\inftysm{n-1}y_n^q}^\frac{1}{q} = \nm{\Mf x}_q.
}
Consequently, it suffices to show that $\inftysm{n=1} x_n^{q-1} y_n$  can be made arbitrarily close to $\nm{f}_{1,\Qplus}$. We have
\eqns{
\inftysm{n=1}x_n^{q-1}y_n &= \frac{1}{d(c)^{\frac{q-1}{q}}} \sum_{n \mid c} y_n =  \frac{1}{d(c)^{\frac{q-1}{q}}} \sum_{n \mid c} \sum_{k \mid c} f\bracks{\frac{n}{k}} x_k  \\
&= \frac{1}{d(c)} \sum_{n,k \mid c} f\bracks{\frac{n}{k}}.
}
We now follow the argument given in \cite{Hilberdink2015} (page 87). For $s =\frac{u}{v} \in \Qplus$,
\eqns{
\frac{1}{d(c)} \sum_{n,k \mid c} f\bracks{\frac{n}{k}} =  \frac{1}{d(c)} \sum_{s \in \Qplus} f(s) \sum_{\substack{n,k \mid c \\ s=\frac{n}{k}}} 1  = \frac{1}{d(c)} \sum_{u,v \in \N} f\bracks{\frac{u}{v}} \sum_{\substack{n,k \mid c \\ nv = uk }} 1,
}
where we used that $\frac{n}{k} = \frac{u}{v}$ if and only if  $nv = uk$. Since $(u,v)=1$ we have  $u \mid n$ and $v \mid k$,  and for any contribution to the summation on the RHS, we must have $u,v \mid c$, i.e., $uv \mid c$. Assume therefore, that $uv \mid c$. By writing $n=lu$ and $k=lv$ for some $l \in \N$, we get
\begin{eqnarray*}
& & \frac{1}{d(c)}\sum_{u,v \in \N} f\bracks{\frac{u}{v}} \sum_{\substack{n,k \mid c \\ nv = uk }} 1
= \frac{1}{d(c)} \sum_{u,v \in \N} f\bracks{\frac{u}{v}} \sum_{lu, lv  \mid c } 1\\
& & = \frac{1}{d(c)} \sum_{u,v \in \N} f\bracks{\frac{u}{v}} \sum_{l  \mid \frac{c}{uv} } 1
 =   \sum_{u,v \in \N} f\bracks{\frac{u}{v}} \frac{d\bracks{c/uv }}{d(c)} .
\end{eqnarray*}
Now, by choosing $c$ appropriately, we can show that  $\frac{d\bracks{c/uv }}{d(c)} $ can be made close to 1 for all $u,v$ less than some large constant. Fix $T \in \Prime$ and choose $c$ to be
\eqns{
c = \prod_{\substack{t \leq T \\ t \in \Prime}} t ^{\alpha_t} \;\:\text{ where }\;\: \alpha_t = \left[\frac{\log T}{\log t}\right].
}
If $uv  \mid c $, then $uv  = \prod_{t \leq T} t^{\beta_t}$ where $\beta_t \in [0,\alpha_t]$ and hence
\eqns{
\frac{d\bracks{c/uv }}{d(c)} = \prod_{t \leq T} \bracks{\frac{\alpha_t - \beta_t +1}{\alpha_t + 1}} = \prod_{t \leq T} \bracks{1 - \frac{\beta_t}{\alpha_t + 1}}.
}
If we take $uv \leq \sqrt{\log T}$, then $t^{\beta_t} \leq \sqrt{\log T}$ for every prime divisor $t$ of $uv $. Therefore, $\beta_t \leq \frac{\log \log T}{2\log t}$ and $\beta_t = 0$ if $t > \sqrt{\log T}.$ It follows that
\eqns{
\frac{d\bracks{c/uv}}{d(c)} &= \prod_{t \leq \sqrt{\log T}} \bracks{1 - \frac{\beta_t}{\alpha_t + 1}} \geq \prod_{t \leq \sqrt{\log T}} \bracks{1 - \frac{\log \log T}{2 \log T}} \\
&= \bracks{1 - \frac{\log \log T}{2 \log T}}^{\pi\bracks{\sqrt{\log T}}},
}
where $\pi\bracks{x}$ is the prime counting function up to $x$. As $\pi(x) \ll \frac{x}{\log x}$, we have for sufficiently large $T$,
\eqns{
 \frac{d\bracks{c/uv }}{d(c)} = \bracks{1 - \frac{\log \log T}{2 \log T}}^{\pi\bracks{\sqrt{\log T}}} \geq 1 - \frac{C}{\sqrt{\log T}},
}
for some  constant $C$. Therefore,
\eqns{
\sum_{u,v \in \N} f\bracks{\frac{u}{v}} \frac{d\bracks{c/uv }}{d(c)}
&>  \sum_{uv \leq \sqrt{\log T}} f(s) \bracks{1 - \frac{C}{\sqrt{\log T}}}   - \sum_{uv  > \sqrt{\log T}} f(q)  \\
 &\geq \sum_{s \in \Qplus} f(s) - \frac{C_1}{\sqrt{\log T}} - 2\sum_{uv  > \sqrt{\log T}} f(s),
}
as $f \in \el^1(\Qplus)$. By choosing $T$ to be arbitrarily large, for every $\epsilon > 0$, we have
\eqns{
\nm{f}_{1,\Qplus} - \nm{\Mf x}_q \leq \nm{f}_{1,\Qplus} - \inftysm{n=1} x_n^{q-1} y_n <  \epsilon.
}
\noindent
$\bullet$ We now consider the case where $p= q = \infty$, and so $r=1$.

\medskip
\noindent
 Let $x_n=1$ for all $n \in \N$ so that $\nm{x}_\infty = 1$. Moreover, for a fixed $c \in \N$, we have
\eqns{
\modu{y_c} = \inftysm{k=1}f\bracks{\frac{c}{k}} x_k = \inftysm{k=1} f\bracks{\frac{c}{k}}.
}
Again, by applying the same methods already shown, we conclude that $y_c$ can be arbitrarily close to $\nm{f}_{1,\Qplus}$. Hence, $\nm{\Mf}_{\infty,\infty} = \nm{f}_{1,\Qplus}$.

\medskip
\noindent
3. Finally, we consider the case when $q= \infty$ with any $p$. We have already dealt with the case when $p=1$ and $p=\infty$. So let $p \in (1,\infty)$, giving $r=\frac{p}{p-1}$.

\medskip
\noindent
Fix $c \in \N$, and let
\eqns{
x_n = f\bracks{\frac{c}{n}}^\frac{r}{p}F_c^{-\frac{1}{p}} \text{ where } F_c = \inftysm{n=1} f\bracks{\frac{c}{n}}^r \text{  exists as } f \in \el^r(\Qplus).
}
With this choice,
\eqns{
\nm{x}_p = \frac{1}{F_c} \inftysm{n=1} f\bracks{\frac{c}{n}}^r = \frac{F_c}{F_c} =1.
}
Now consider just the term $y_c$,
\eqns{
y_c = F_c^{-\frac{1}{p}} \inftysm{k=1} f\bracks{\frac{c}{k}}   f\bracks{\frac{c}{k}}^\frac{r}{p} =  F_c^{-\frac{1}{p}} \inftysm{k=1} f\bracks{\frac{c}{k}}^r,
}
as $1 + \frac{r}{p} = \frac{p-1 + 1}{p-1} = r$. Therefore,
\eqns{
y_c=  F_c^{1-\frac{1}{p}} = F_c^{\frac{1}{r}} = \bracks{\inftysm{k=1} f\bracks{\frac{c}{k}}^r}^{\frac{1}{r}}.
}
We can apply the same argument as before to show that for every $\epsilon > 0$, we can choose $c= (2 \cdot 3 \cdot 5 \cdots T)^k$ where $T$ is prime such that $y_c$ can be made arbitrarily close to $\nm{f}_{r,\Qplus}$. Hence, $\nm{\Mf}_{p,\infty} = \nm{f}_{r,\Qplus}$.
\end{proof}

\begin{rem} In \cite{Hilberdink2015}, the author showed that if $f$ is any, not necessarily strictly positive sequence, in $\el^1(\Qplus)$, then $\Mf: \eltwo \to \eltwo$ is bounded and the operator norm is given by
\eqns{
\nm{\Mf}_{2,2} = \sup_{t \in \R}\Bigg|\sum_{q \in \Qplus} f(q) q^{it}\Bigg|.
}
By assuming $f$ positive, the supremum of the above is attained when $t=0$, and as such $\nm{\Mf}_{2,2} = \nm{f}_{1,\Qplus}$ as given in Theorem \ref{thm:operatornormedgecases}. The differing operator norm when $f$ is not positive, is echoed in the work of \cite{codeca2008}, where an example is given showing that $\nm{\Df}_{p,p} \not = \nm{f}_{1}$.  Determining $\nm{\Mf}_{p,q}$ for any $f$ and general $p,q$ remains an open question, but is not however the focus of this paper.
\end{rem}

  \section{Connection with multiplicative sequences}

                      Generalising Theorem \ref{thm:operatornormedgecases} to find a necessary condition and the operator norm for all other $p$ and $q$ (which we will refer to as the interior cases) is challenging and is the focus of the proceeding discussions.

We start by taking $f$ supported on $\N$, i.e., $\Mf=\Df$. To understand the behaviour of the operator norm in the interior cases, we can consider where $\nm{\Df x}$ attains its supremum value in the edge cases.  First, setting $c=1$ in case 1 of the proof of Theorem \ref{thm:operatornormedgecases} yields the supremum of $\nm{\Df x}_{q}$. This gives $x_n = 1$ if $n =1$ and $0$ otherwise and as such $x$ is completely multiplicative. Secondly, for $1<p = q < \infty$ in case 2, we choose $x_n = \frac{1}{d(c)^{1/p}}$, whenever $n \mid c$, and 0 otherwise, which is a constant multiplicative sequence. Moreover, for $p=q=\infty$, the completely multiplicative sequence $x_n = 1$ (for all $n \in \N$) attains the operator norm. Finally, in case 3, for $f$ multiplicative, $x$ is again constant multiplicative.

It follows, for the edge cases, that $\Df$ is ``largest" when acting on a sequence $x \in \elp$ that has multiplicative structure.  Why this is the case is unclear and leads to a surprising connection between the operator norm of $\Df$ and the set of multiplicative elements in $\elp$, which we denote by $\Mp$. Moreover, we shall denote the set of completely multiplicative sequences in $\elp$ by $\Mcp$.  It is interesting to ask therefore how $\Df$ acts on these sets for $1 < p < q < \infty$, as from this connection, we would expect $\Df :\elp \to \elq$ to attain its supreme value here. Thus, we shall investigate the boundedness of $\Df: \Mcp \to \elq$ for $1 < p < q <\infty$, with the aim of giving some insight into $\nm{\Df}_{p,q}$ \footnote{$\Mcp$ and $\Mp$ are subsets, not subspaces of $\elp$. For example, they are not closed under addition. Given $X, Y$ which are subsets of some Banach space, we say $L:X \to Y$ is bounded $\iff \nm{Lx} \leq C\nm{x}$ for all $x \in X.$}.

From Theorem \ref{thm:boundedonlp}, it follows that $\Df: \Mcp \to \elq$ is bounded if $f \in \el^r$. We wish to know whether this is also a necessary condition. In Theorem \ref{thm:boundedonmt}, we  show that for $f$ completely multiplicative, the requirement that $f$ be $\Mc^r$ is not a necessary condition for $\Df: \Mcp \to \Mtwo$ to be bounded\footnote{The convolutions of two multiplicative sequences is also multiplicative, so we can consider $y \in \Mtwo$. } when $p \in (1,2)$ and $q=2$. One can speculate therefore that $f \in \el^r$ is not a necessary condition when considering $\Df:\elp \to \eltwo$.

  \begin{thm}
  \label{thm:boundedonmt}
Let $1 < p < 2$.  If $f \in \Mctwo$, the mapping $\Df:\Mcp \to \Mtwo$ is bounded.
 \end{thm}

To highlight the difference between this criterion and that shown in the previous section, we consider the following example. Let $f(n) = \frac{1}{n^\alpha}$ for $\alpha >\frac{1}{2}$  and let $p = \frac{3}{2}$, giving $\frac{1}{r} = 1 -\frac{2}{3} + \frac{1}{2} = \frac{5}{6}$. Theorem \ref{thm:boundedonlp} states that if $\alpha >  \frac{5}{6}$, then $\Df:\Mc^\frac{3}{2} \to \eltwo$ is bounded. In contrast, Theorem \ref{thm:boundedonmt} shows that only $\alpha > \frac{1}{2}$ is required for boundedness. For the proof of Theorem \ref{thm:boundedonmt}, we will require the following lemma, which will be proved below.

 \begin{lem}
  Let $f,g,h,j \in \Mctwo$. Then,
  \eqnslab{
  \label{eqn:innerproductconvolution}
  \ip{f\ast g, h \ast j} &=  \frac{\ip{g,j}\ip{f,h} \ip{f,j}\ip{g,h}}{\ip{fg,hj}}.
  }
   \label{lem:productconv}
  \end{lem}

\begin{proof}[Proof of Theorem \ref{thm:boundedonmt}]
By taking $h=f$ and $g =j = x$ in (\ref{eqn:innerproductconvolution}), we have
\eqns{
\nm{\Df x}_2 = \frac{\nm{f}_2 \nm{x}_2 \modu{\ip{f,x}}}{\nm{fx}_2}  \leq \nm{f}_2  \nm{x}_2 \modu{\ip{f,x}},
}
as $f$ and $x$ are multiplicative, and as such we have $x_1=1$ and $f(1) = 1$, giving
\eqns{
\nm{fx}_2 = \inftysm{n=1} \modu{f(n)x_n}^2 \geq 1.
}
Now,
\eqnslab{\label{eqn:productofprimes}
\frac{\nm{\Df x}_2}{\nm{x}_p} &\leq  \frac{ \nm{f}_2 \nm{x}_2 \modu{\ip{f,x}}} {\nm{x}_p} = \nm{f}_2\prod_{t \in \Prime} \frac{\bracks{1-\modu{x_t}^p}^{\frac{1}{p}}}{\bracks{1-\modu{x_t}^2}^{\frac{1}{2}}\bracks{1-\modu{x_tf(t)}}},
}
where we made use of Euler products. Therefore, it remains to show that the product over primes is bounded independently of $x_t$. As $0 \leq \modu{x_t} < 1$, we can say that
\eqns{
\modu{x_t}^2 < \modu{x_t}^p \implies \frac{1}{1-\modu{x_t}^2} < \frac{1}{1-\modu{x_t}^p}.
}
Hence, the product of (\ref{eqn:productofprimes}) is at most
\eqns{
\prod_{t \in \Prime} \frac{\bracks{1-\modu{x_t}^p}^{\frac{1}{p}}}{\bracks{1-\modu{x_t}^p}^{\frac{1}{2}}\bracks{1-\modu{x_t f(t)}}} =  \prod_{t \in \Prime} \frac{\bracks{1-\modu{x_t}^p}^{\frac{2-p}{2p}}}{\bracks{1-\modu{x_t f(t)}}}.
}
By taking logarithms, we arrive at the equality
\eqns{
\log \bracks{\prod_{t \in \Prime} \frac{\bracks{1-\modu{x_t}^p}^{\frac{2-p}{2p}}}{\bracks{1-\modu{x_t f(t)}}}} =  \sum_{t \in \Prime} \bracks{\log \frac{1}{1-\modu{x_t f(t)}} - \frac{2-p}{2p}\log\frac{1}{1-\modu{x_t}^p} } .
}
Note in general for $a >0$, we have $ a \leq \log\bracks{\frac{1}{1-a}} = a + O\bracks{a^2}$. Hence,
\eqns{
\sum_{t \in \Prime} \log \bracks{\frac{1}{1-\modu{x_t}^p}} \geq \sum_{t \in \Prime} \modu{x_t}^p,
}
and moreover,
\eqns{
\sum_{t \in \Prime} \log \bracks{\frac{1}{1-\modu{x_t f(t)}}} = \sum_{t \in \Prime} \modu{x_t f(t)} + O(1),
}
where the $O(1)$ term is independent of the sequence $x_t$. Therefore, we obtain
\eqns{
&\sum_{t \in \Prime} \bracks{\log \frac{1}{1-\modu{x_t f(t)}} - \frac{2-p}{2p}\log\frac{1}{1-\modu{x_t}^p} } \\
 &<  \sum_{t \in \Prime} \bracks{\modu{x_t f(t)} - \frac{2-p}{2p}\modu{x_t}^p} + O(1).
}
Now, we consider the case when the terms of the above series are positive. In other words,
\eqns{
\modu{x_t f(t)} \geq \frac{2-p}{2p} \modu{x_t}^p \iff \bracks{\frac{2p}{2-p} \modu{f(t)} }^{\beta} \geq \modu{x_t},
}
where $\beta = \frac{1}{p-1}$. Hence, by only summing over the $t$ which yield positive terms, we have
\begin{eqnarray*}
& & \sum_{t \in \Prime} \bracks{\modu{x_t f(t)} - \frac{2-p}{2p}\modu{x_t}^p} < \sum_{\substack{t \text{ s.t} \\ \modu{x_t} \leq \bracks{\frac{2p}{2-p} \modu{f(t)}}^{\beta}}} \bracks{\modu{x_t f(t)} - \frac{2-p}{2p}\modu{x_t}^p }\\
& &\leq \sum_{\substack{t \text{ s.t} \\ \modu{x_t} \leq \bracks{\frac{2p}{2-p} \modu{f(t)}}^{\beta}}} \modu{x_t f(t)}
\leq \bracks{\frac{2p}{2-p}}^\beta \sum_{t \in \Prime}  \modu{f(t)}^\beta \modu{f(t)}.
\end{eqnarray*}
As $\beta + 1 = \frac{p}{p-1} > 2$, we see that
\eqns{
\sum_{t \in \Prime}  \modu{f(t)}^{\beta + 1} \leq \sum_{t \in \Prime}  \modu{f(t)}^{2}  < \infty,
}
as $f \in \Mctwo$. Hence,  the product in (\ref{eqn:productofprimes}) is bounded, which implies that the mapping $\Df : \Mcp \to \Mtwo$ is bounded.
\end{proof}

  \begin{proof}[Proof of Lemma \ref{lem:productconv}]
We start by computing the LHS of (\ref{eqn:innerproductconvolution}):
\begin{eqnarray*}
& & \ip{f\ast g, h \ast j} = \sum_{n \geq 1} (f\ast g)(n)\overline{(h \ast j) (n)} =
\sum_{n \geq 1} \sum_{c,d|n} f(c)g\left(\frac{n}{c}\right)\overline{ h(d)j\left(\frac{n}{d}\right) }\\
& & = \sum_{c,d \geq 1} \sum_{\substack{n \geq 1 \\ c,d|n}} f(c)g\left(\frac{n}{c}\right) \overline{h(d)j\left(\frac{n}{d}\right)}
= \sum_{c,d \geq 1} \sum_{\substack{n \geq 1 \\ [c,d]|n}} f(c)g\left(\frac{n}{c}\right) \overline{h(d)j\left(\frac{n}{d}\right) },
\end{eqnarray*}
since $c,d|n \iff [c,d]|n$. Now, as $[c,d]|n \iff n =[c,d]m$, the above is given by
\begin{align*}
&\sum_{c,d \geq 1} \sum_{m \geq 1} f(c)g\left(\frac{m[c,d]}{c}\right)\overline{ h(d)j\left(\frac{m[c,d]}{d}\right)}  \\
&=  \sum_{m \geq 1} g(m)\overline{j(m)}\sum_{c,d \geq 1}  f(c)g\left(\frac{[c,d]}{c}\right)\overline{ h(d)j\left(\frac{[c,d]}{d}\right)} \\
&=  \ip{g,j}\sum_{c,d \geq 1}  f(c)g\left(\frac{d}{(c,d)}\right)\overline{ h(d)j\left(\frac{c}{(c,d)}\right)}  \qquad \text{ as } [c,d](c,d)=cd \\
&=  \ip{g,j} \sum_{k \geq 1}\sum_{\substack{c,d \geq 1 \\ (c,d) = k}}  f(c)g\left(\frac{d}{(c,d)}\right) \overline{h(d)j\left(\frac{c}{(c,d)}\right) }.
\end{align*}
If $(c,d) = k$, then $c = c'k, d = d'k$ where $(c',d')=1$. Therefore,
\[
\ip{f\ast g, h \ast j} =  \ip{g,j} \sum_{k \geq 1}\sum_{\substack{c',d' \geq 1 \\ (c',d') = 1}}  f(c'k)g(d') \overline{ h(d'k)j(c')} \nn,\]
which is equal to
\begin{eqnarray}
& & \ip{g,j} \sum_{k \geq 1} f(k) \overline{h(k)}\sum_{\substack{c',d' \geq 1 \\ (c',d') = 1}}  f(c')g(d')\overline{ h(d')j(c') } \nn \\
& & =  \ip{g,j}\ip{f,h}\sum_{\substack{c',d' \geq 1 \\ (c',d') = 1}}  f(c')g(d')\overline{ h(d')j(c')}.  \label{eqn:innerproductLHS}
\end{eqnarray}
We now compute the RHS of (\ref{eqn:innerproductconvolution}). We have
\eqnslab{
\ip{f,j}\ip{g,h} &= \sum_{c,d \geq 1} f(c)\overline{j(c)}g(d)\overline{h(d)} = \sum_{k\geq 1}\sum_{\substack{c,d \geq 1 \\ (c,d)=k}} f(c)\overline{j(c)}g(d)\overline{h(d)} \nn \\
&= \sum_{k\geq 1}\sum_{\substack{c',d' \geq 1 \\ (c',d')=1}} f(c'k)\overline{j(c'k)}g(d'k)\overline{h(d'k)} \nn \\
&= \sum_{k\geq 1} f(k)\overline{j(k)}g(k)\overline{h(k)}\sum_{\substack{c',d' \geq 1 \\ (c',d')=1}} f(c')\overline{j(c')}g(d')\overline{h(d')} \nn \\
&= \ip{fg,hj}\sum_{\substack{c',d' \geq 1 \\ (c',d')=1}} f(c')\overline{j(c')}g(d')\overline{h(d')}.  \label{eqn:innerproductRHS}
}
Hence, by comparing (\ref{eqn:innerproductLHS}) with (\ref{eqn:innerproductRHS})  we obtain (\ref{eqn:innerproductconvolution}).
\end{proof}

Naturally one can ask if Theorem \ref{thm:boundedonmt} generalises to $\elp$. In other words: {\em is $f \in \eltwo$ a sufficient condition for  $\Df:\elp \to \eltwo$
to be bounded for every  $p$ in $(1,2)$?}
Theorem \ref{thm:boundedonmt} raises some interesting points of speculation regarding this question. It would perhaps be surprising if Theorem \ref{thm:boundedonmt} could not be generalised to $\Mf$ on $\elp$ as we know that in the edge cases, the operator norm is ``largest" when acting on multiplicative sequences. Why this would not also be true for the interior cases is unclear. In contrast, we know from Theorem \ref{thm:operatornormedgecases} that when $p=2$, $f \in \el^1$ is needed for boundedness. If a generalisation were possible, there would be a jump in the required value of $r$. That is, by considering $p = 2 -\epsilon$ for any $\epsilon >0$, $f \in \eltwo$ is all that is required. Why the jump between $f \in \el^1$ to $f \in \eltwo$ would occur is also unclear.
Finding a generalisation of Theorem \ref{thm:boundedonmt} has not been possible, and  leads to an investigation of a possible counterexample to the question
raised above.

\subsection*{A possible counterexample}

We wish to know, given $f \in \eltwo$, does there exist $x \in \elp$, for $p \in (1,2)$, such that $\frac{\nm{\Df x}_p}{\nm{x}_2}$  can be arbitrarily large? For simplicity, we choose $f(n) = \frac{1}{n^\alpha}$ with $\alpha > \frac{1}{2}$.

  \begin{prop} \label{thm:counterexample} Let $p \in (1,2)$, $q=2$, and $\alpha > \frac{1}{2}$. If $(x_n) \in \elp$ is a sequence such that
$x_n \ll {1}/{d(n)^{\frac{1}{2-p}}}$, then $\Da x \in \eltwo$.
  \end{prop}

  \begin{proof}
By the Cauchy-Schwarz inequality, we have
\eqns{
y_n^2 &= \bracks{\sum_{d \mid n} \frac{x_{\sfrac{n}{d}}}{d^\alpha}}^2 \leq \sum_{d \mid n }  1 \sum_{d \mid n }  \frac{x^2_{\sfrac{n}{d}}}{d^{2\alpha}} = d(n) \sum_{d \mid n} \frac{x^2_{\sfrac{n}{d}}}{d^{2\alpha}}.
}
So,
\eqns{
\nm{\Da x}_2^2 &\leq \inftysm{n=1} d(n) \sum_{d \mid n} \frac{x^2_{\sfrac{n}{d}}}{d^{2\alpha}} = \inftysm{d=1} \inftysm{m=1} d(md)  \frac{x_m^2}{d^{2\alpha}}  \quad \text{ by writing } dm=n \\
&\leq \inftysm{d=1}  \frac{d(d)}{d^{2\alpha}} \inftysm{m=1} d(m) x_m^2,
}
as $d(mn) \leq d(m)d(n)$. As $\alpha > \frac{1}{2}$, the first series on the RHS is convergent (and given by $\zeta(2\alpha)^{2}$). Hence, \eqns{\nm{\Da x}_2^2 \ll \inftysm{m=1} d(m) x_m^2.} This is convergent if $x_m^2 d(m)  \ll x_m^p$ (as $x \in \elp$). By rearranging, this is equivalent to $x_m \ll {1}/{d(m)^{\frac{1}{2 - p}}}$ as required.
\end{proof}

  From Proposition \ref{thm:counterexample}, we can conclude that any counterexample, say $x = (x_n)$, must satisfy $x_n > 1/d(n)^{\frac{1}{2 - p}}$ for infinitely many $n \in \N$.  As such we define
  \eqns{
  S = \curlybracks{n \in \N : x_n > {1}/{d(n)^{\frac{1}{2 - p}}} },
  }
  and we may assume that the support of $x$ is contained within the set $S$, i.e., $x_n = 0$ if $n \notin S$. However, some care must be taken in choosing $S$ (if an example is possible) as
  \eqnslab{\label{eqn:criteriaonsets}
  \sum_{n \in S }\frac{1}{d(n)^{\frac{p}{2 - p}}} \leq \sum_{n \in S } x_n^p < \infty,
  }
  must be satisfied as $x \in \elp$. First, $S$ must be a ``sparse" set; consider the function which counts the number of $n \in S$ below a given $x$, $S(x) = \sum_{\substack{n \leq x \\ n \in S}} 1$. Then
  \eqns{
  S(x) = \sum_{\substack{n \leq x \\ n \in S}} \frac{x_n^p}{x_n^p} \ll x^\epsilon \sum_{\substack{n \leq x \\ n \in S}} x_n^p \ll x^\epsilon \text{ for all } \epsilon > 0,}
as ${1}/{x_n^p} \leq d(n)^{\frac{p}{2-p}} \ll n^\epsilon \leq x^\epsilon$ for all $\epsilon > 0$. For example, choosing $S = \N$ fails. Secondly, $S$ must contain $n$ with large numbers of divisors, otherwise $1/d(n)^{\frac{p}{2-p}} \not \to 0$ as $n \to \infty$ and so (\ref{eqn:criteriaonsets}) will not be satisfied ($S$ can not be a subset of $\Prime$, for example). However, the following example indicates the difficulty of choosing $S$ to yield $\Da$ unbounded: define $S = \curlybracks{2^k : k \in \N}$. We see that (\ref{eqn:criteriaonsets}) is satisfied because
  \eqns{   \sum_{n \in S }\frac{1}{d(n)^{\frac{p}{2 - p}}}  = \inftysm{k=1 }\frac{1}{\bracks{k+1}^{\frac{p}{2 - p}}} < \infty
  \;\text{ as } \;\frac{p}{2-p} > 1 \;\text{ for } \; p \in (1,2).
  }
Now,
  \eqns{
  y_n = \sum_{2^k \mid n } \frac{2^{k\alpha}}{n^\alpha} x_d .
  }
  Write $n = 2^l m$ where $m$ is odd. Then
  \eqns{
  \bracks{y_{2^l m}}^2 &= \bracks{\sm{k=0}{l} \frac{x_{2^k}}{\bracks{2^{l-k} m} ^\alpha}}^2 = \frac{1}{m^{2\alpha}}\bracks{\sm{k=0}{l} \frac{x_{2^k}}{2^{\bracks{l-k}\alpha}}}^2 \\
  &= \frac{1}{m^{2\alpha}}\bracks{\sm{k=0}{l} \frac{x_{2^{l-k}}}{2^{k\alpha}}}^2  \text{ by writing } k \mapsto l - k \\
  &= \frac{1}{m^{2\alpha}} \bracks{\sm{k=0}{l} \frac{x_{2^{l-k}}}{2^{k(\alpha - \delta)}} \frac{1}{2^{k\delta}}}^2 \leq \frac{1}{m^{2\alpha}} \sm{k=0}{l}\bracks{ \frac{x_{2^{l-k}}}{2^{k(\alpha - \delta)}}}^2 \sm{k=0}{l} \frac{1}{2^{2k\delta}}  \\
  &\ll \frac{1}{m^{2\alpha}} \sm{k=0}{l}\bracks{ \frac{x_{2^{l-k}}}{2^{k(\alpha - \delta)}}}^2 .
  }
  We now sum over all $l$ and $m$,
  \eqns{
  \inftysm{l=1} \sum_{\substack{m \in \N \\ m \text{ odd}}}\bracks{y_{2^l m}}^2  &\ll \inftysm{l=1} \sum_{\substack{m \in \N \\ m \text{ odd}}} \frac{1}{m^{2\alpha}} \sm{k=0}{l}\bracks{ \frac{x_{2^{l-k}}}{2^{k(\alpha - \delta)}}}^2 \\
 &\leq \zeta(2\alpha) \inftysm{l=0} \sm{k=0}{l}\bracks{ \frac{x_{2^{l-k}}}{2^{k(\alpha - \delta)}}}^2 \\
  &\ll  \inftysm{k=0} \inftysm{l = 0}\bracks{ \frac{x_{2^{l}}}{2^{k(\alpha - \delta)}}}^2  = \inftysm{k=0} \frac{1}{2^{2k(\alpha - \delta)}} \inftysm{l = 0} x_{2^{l}}^2 ,
  }
 which is finite as $x \in \elp$.   The following Proposition suggests some further structure of $S$.
 \begin{prop} Let $\alpha > \frac{1}{2}$ and $\beta = \frac{p}{(2-p)(2\alpha -1)}$. Let $y= \gamma + \mu$ where $\gamma=(\gamma_n)$ and $\mu=(\mu_n)$ are given by
\label{thm:orderofxcounterexample}
\eqns{
\gamma_n = \sum_{\substack{d \mid n \\ \frac{n}{d} \in S \\ d \geq d(n)^\beta}}  \frac{x_{n/d}}{d^\alpha} \quad \text{ and } \mu_n = \sum_{\substack{d \mid n \\ \frac{n}{d} \in S  \\ d < d(n)^\beta}} \frac{x_{n/d}}{d^\alpha} .
}
Then $\gamma \in \eltwo$.
\end{prop}

\begin{proof}
By the Cauchy-Schwarz inequality,
\eqns{
\gamma_n^2   &= \Bigg ( \sum_{\substack{ d \mid n \\ \frac{n}{d} \in S  \\ d \geq d(n)^\beta }}  \frac{x_{n/d}}{d^\alpha} \Bigg )^2 = \Bigg (\sum_{\substack{ d \mid n \\ d \in S  \\ d \leq \frac{n}{d(n)^\beta} }} x_{d} \bracks{\frac{d}{n}}^\alpha \Bigg ) ^2 \\
&\leq \sum_{\substack{ d \mid n \\ d \in S  \\ d \leq \frac{n}{d(n)^\beta}}} x_{d}^2  \sum_{\substack{ d \mid n \\ d \in S  \\ d \leq \frac{n}{d(n)^\beta}}} \bracks{\frac{d}{n}}^{2\alpha} \ll \sum_{\substack{ d \mid n \\ d \in S  \\ d \leq \frac{n}{d(n)^\beta}}} \bracks{\frac{d}{n}}^{2\alpha},
}
as $x \in \eltwo$. Therefore,
\eqns{
\inftysm{n=1} \gamma_n^2 &\ll \inftysm{n=1} \sum_{\substack{ d \mid n \\ d \in S  \\ d \leq \frac{n}{d(n)^\beta}}} \bracks{\frac{d}{n}}^{2\alpha} \leq \sum_{d \in S} \sum_{\substack{m \geq 1 \\ d(dm)^\beta < m}} \frac{1}{m^{2\alpha}} \\
&\leq \sum_{ d \in S } \sum_{d(d)^\beta< m} \frac{1}{m^{2\alpha}} \ll \sum_{ d \in S} \frac{1}{d(d)^{\beta\bracks{2\alpha -1}}},
}
as, for $s > 1$,  \eqns{
\sum_{n > m } \frac{1}{n^s} \ll m^{1-s},
}
(see \cite{Apostol1976}, page 55). By assumption, we have

\eqns{
\sum_{ n \in S} \frac{1}{d(n)^{\beta\bracks{2\alpha -1}}} = \sum_{n \in S }\frac{1}{d(n)^{\frac{p}{2 - p}}} \leq \sum_{n \in S } x_n^p < \infty,
}
as required.
  \end{proof}

From Proposition \ref{thm:orderofxcounterexample}, we can see that any counterexample must yield $\mu \not \in \eltwo$. This suggests that $S$ must contain $n \in \N$ such that $n$ has a large number of small divisors so that $d < d(n)^\beta$ is satisfied often and in turn ensuring that many divisors contribute to the summation. The investigation of finding a suitable support set $S$ has not yet yielded $\mu \not \in \eltwo$, and this gives little indication of a successful counterexample. The lack of existence of either a generalisation of Theorem \ref{thm:boundedonmt} or a counterexample demonstrates perhaps the challenging nature of this problem and leaves further open questions regarding the boundedness of multiplicative Toeplitz operators.

\subsection*{Open questions}

We conclude this paper by summarising the open problems that have risen from our discussion. \begin{itemize}
\item Is $f \in \el^r$ a necessary condition for $\Df : \elp \to \elq$ to be bounded for any $p$ and $q$?
\item Can we generalise Theorem \ref{thm:boundedonmt} from multiplicative subsets to the mapping $\Df: \elp \to \eltwo$? Or can we find a counterexample to this?
\end{itemize}
Finally, we give some further open questions regarding multiplicative Toeplitz operators which we have not discussed in this paper.

\begin{itemize}
\item What is the operator norm when $f$ can take negative values? Does it mimic that given in \cite{Hilberdink2015}?
\item Can we compute the spectrum of $\Mf$? Does $\Mf$ have any eigenvalues and if so what are they?
\item For which $f$ is $\Mf$ Fredholm, and can we describe the essential spectrum of $\Mf$?
\end{itemize}

\addressprint

\end{document}